\documentclass[12pt]{article}
\usepackage{amsmath,amssymb,color,amscd,centernot}
\usepackage{xspace}

\newcommand{\Var}{{\rm{Var}_{\mathbb{C}}}}

\newcommand{\ind}{{\mbox{\rm ind}}}

\def\1{\underline{1}}
\def\R{{\mathbb R}}

\def\Z{{\mathbb Z}}

\def\RR{{\mathcal R}}

\def\conjsub{{\mbox{conjsub\,}}}
\def\ind{{\rm ind}}

\def\chiun{{\chi^{\rm un}}}
\def\G{{\mathcal{G}}}
\def\GG{{\mathcal{G}}}

\newtheorem{theorem}{Theorem}

\newtheorem{proposition}{Proposition}
\newtheorem{definition}{Definition}

\newenvironment{remark}
{\smallskip\noindent{\bf Remark\/}.}{\smallskip\par}

\newenvironment{proof}
{\noindent{\bf Proof\/}.}{{ $\square$}\smallskip\par}

\title{The universal Euler characteristic of $V$-manifolds
}

\author{S.M.~Gusein-Zade \thanks{The work of the first author
(Sections~\ref{sec:Euler_manifolds}, \ref{sec:Universal_Euler}, \ref{sec:lambda-power})
was supported by the grant 16-11-10018 of the Russian Science Foundation.
Address: Moscow State University, Faculty
of Mechanics and Mathematics, GSP-1, Moscow, 119991, Russia. E-mail:
sabir\symbol{'100}mccme.ru} \and I.~Luengo \thanks{The last two authors were partially
supported by a competitive Spanish national grant MTM2016-76868-C2-1-P.
Address:  ICMAT (CSIC-UAM-UC3M-UCM), Dept. of Algebra, Geometry and Topology, Complutense University of Madrid,
Plaza de Ciencias 3, Madrid, 28040, Spain.
E-mail: iluengo\symbol{'100}mat.ucm.es} \and
A.~Melle-Hern\'andez \thanks{Address:  Instituto de Matem\'a¡tica Interdisciplinar (IMI), Dept. of Algebra, Geometry and Topology, 
Complutense University of Madrid,
Plaza de Ciencias 3, Madrid, 28040, Spain. E-mail: amelle\symbol{'100}mat.ucm.es}}
\date{}
\begin{document}
\def\eps{\varepsilon}

\maketitle

\begin{abstract}

The Euler characteristic is the only additive topological invariant for spaces of certain sort, in particular,
for manifolds with some finiteness properties. A generalization of the notion of a manifold is the notion of
a $V$-manifold. Here we discuss a universal additive topological invariant of $V$-manifolds: the universal
Euler characterictic. It takes values in the ring generated (as a $\Z$-module) by isomorphism classes of finite
groups. We also consider the universal Euler characteristic on the class of locally closed equivariant unions
of cells in equivariant CW-complexes. We show that it is a universal additive invariant satisfying a certain
``induction relation''. We give Macdonald type equations for the universal Euler characteristic for $V$-manifolds
and for cell complexes of the described type.
\end{abstract}

\section{Introduction}\label{sec:Intro}

The Euler characteristic $\chi(\cdot)$ (defined as the alternating sum of the ranks of the cohomology groups with compact
support) is the only \emph{additive} topological invariant for spaces of certain sort: see. e.~g., \cite{Watts},
see also \cite[Proposition~2]{SGZ-2017}. In particular, the Euler characteristic is the only additive invariant
of manifolds with some finiteness properties: see below. This property has some generalizations. For example,
the equivariant Euler characteristic with values in the Burnside ring $A(G)$ of a finite group $G$ is the only 
additive topological invariant of spaces with $G$-actions (see, e.~g., \cite{SGZ-2017}).

A generalization of the notion of a manifold is the notion of a $V$-manifold (that is of a (real) orbifold:
locally defined as the quotient of a manifold by a finite group action) introduced initially in \cite{Satake1}.
There are a number of additive invariants defined for $V$-manifolds, e.~g., the Euler--Satake characteristic:
\cite{Satake2}, the orbifold Euler characteristic: \cite{Dixon_et_al-1}, \cite{Dixon_et_al-2}, \cite{AS}, \cite{HH},
the higher order (orbifold) Euler characteristics: \cite{BrF}, \cite{Tamanoi}, the $\Gamma$--Euler--Satake
characteristic:
\cite{Farsi+}.

Here we discuss the universal additive topological invariant $\chiun$ of $V$-manifolds: a sort of a universal
(topological) Euler characteristic for them. It takes values in the ring $\RR$ generated (as a $\Z$-module)
by isomorphism classes of finite groups. 

We also consider the universal Euler characteristic $\chiun$ on the Grothendieck ring $K_0^{\rm fGr}(\Var)$ of
quasi-projective varieties with finite groups actions (\cite{fGr}) and on the class of {\em equivariant cell
complexes}: locally closed unions  of cells in equivariant CW-complexes in the sense of \cite{TtD}. In the latter
case we show that $\chiun$ is a universal additive invariant satisfying a certain ``induction relation''.

The classical Euler characteristic satisfies the Macdonald equation for the generating series of the Euler
characteristics of the symmetric powers of a topological space:
$$
1+\sum_{n=1}^{\infty} \chi(S^nX)\cdot t^n=(1-t)^{-\chi(X)}.
$$

Also one has a Macdonald type equation for the Euler characterictics of the configuration spaces of points on $X$.
Let $M_n X= (X^n\setminus \Delta)/S_n$ be the configuration space of (unorderd) 
$n$-tuples of points in $X$ ($\Delta$ is the big diagonal in the Cartesian  power in $X^n$ consisting
of points $\bar{x}=(x_1, \ldots x_n)\in X^n$ with at least two coinciding components). One has
$$
1+\sum_{n=1}^{\infty} \chi(M_nX)\cdot t^n=(1+t)^{\chi(X)}.
$$

Analogues of these  equations for other (additive) invariants with values in rings different from the ring 
of integers (say, for the equivariant Euler characteristic or for the generalized (motivic) Euler characteristic
of complex quasi-projective varieties) are formulated in terms of power structures over the rings of values:
\cite[Lemma 1]{GLM-MMJ-2016}, \cite{GLM-Stek}. A power structure over a ring is closely related with
(and defined by) a $\lambda$-ring structure on it. 
Analogues of these  equations for the universal Euler characterictic are formulated in terms  of different
$\lambda$-ring structures on $\RR$.
We discuss these $\lambda$-ring structures on $\RR$ and the corresponding  power structures. We give
Macdonald type equations for the universal Euler characteristic $\chiun$ for $V$-manifolds and for equivariant
cell complexes.

\section{Euler characteristic of manifolds}\label{sec:Euler_manifolds}

The Euler characteristic is defined for manifolds with some finiteness properties. To fix a class of such manifolds, let us
consider ($C^{\infty}$-) manifolds which are interiors of compact manifolds with boundaries. A submanifold of
such a manifold is the interior of a (closed) submanifold in a manifold with boundary, that is of a submanifold
transversal to the boundary. (We permit a submanifold to be of the same dimension as the manifold itself. In this
case the submanifold is a connected component of the manifold.) In what follows we consider only manifolds from
this class. Let $M$ be a manifold and let $N$ be a (closed) submanifold of $M$. One has the following additivity
property of the Euler characteristic:
$$
\chi(M)=\chi(N)+\chi(M\setminus N)\,.
$$
(Pay attention that $M\setminus N$ is also a manifold from the described class.) One has the inverse statement.

\begin{proposition}\label{prop:unique_for_manifolds}
 Let $I$ be a topological invariant of manifolds which possesses the additivity property:
$$
I(M)=I(N)+I(M\setminus N)
$$
for a submanifold $N\subset M$. Then $I(M)=\chi(M) a$, where $a=I(pt)$.
\end{proposition}

\begin{proof}
 First let us reduce the statement to the corresponding statement for cells, i.~e., for manifolds diffeomorphic
 to open balls. For that we will cut a given $n$-dimensional manifold $M^n$ by submanifolds into pieces diffeomorphic
 to cells. A one-dimensional manifold is a (finite) union of open segments and circles and there is no problem
 to decompose it into cells. Assume that this is possible for manifolds of dimension less than $n$. Let $M$ be
 the interior of a manifold $\widehat{M}$ with boundary and let $f:\widehat{M}\to\R$ be a Morse function on
 $\widehat{M}$ equal to zero on the boundary $\partial\widehat{M}$ and positive on $M$. Let
 $0<c_1<c_2<\ldots<c_r$ be the critical values of $f$ (and let $c_0=0$). Let $U_i$ be small open balls around
 the corresponding critical points $P_i$. The manifold $M$ is the union of the manifolds
 $M_i=f^{-1}((c_{i-1}+\eps, c_i-\eps))$, $M_i'=f^{-1}((c_i-\eps, c_i+\eps))$ and
 $N_{i\pm}=f^{-1}(c_i\pm\eps)$ (we take $\eps$ small enough). We have to cut these manifolds
 into cells (using submanifolds). For the manifolds $N_{i\pm}$ this is possible because of the assumption.
 The manifold $M_i$ is (diffeomorphic to) the cylinder over the manifold $N_{i-}$ and a method to cut $N_{i-}$
 into cells (by submanifolds) can be extended to $M_i$ in an obvious way. The intersection $M_i'\cap U_i$ is
 diffeomorphic to a cell. The complement $M_i'\setminus U_i$ is (diffeomorphic to) the cylinder over
 $N_{i-}\setminus U_i$. A method to cut $N_{i-}\setminus U_i$ gives a method to cut $M_i'\cap U_i$.
 (Here we apply an obvious version of the procedure to $N_{i-}\setminus U_i$ which is a manifold with boundary.)
 
 The additivity property permits to prove the statement for cells (open balls):
 \begin{equation}\label{eq:I_of_cell}
  I(\sigma^k)=(-1)^k I({\rm pt})\,.
 \end{equation}
 Assume that~(\ref{eq:I_of_cell}) is proved for cells of dimension less than $k$, in particular,
 $I(\sigma^{k-1})=(-1)^{k-1}I({\rm pt})$. The ball $\sigma^k$ can be cut by a submanifold diffeomorphic to
 $\sigma^{k-1}$ into two manifolds diffeomorphic to $\sigma^k$. Therefore
 $$
 I(\sigma^k)=2I(\sigma^k)+I(\sigma^{k-1})\,,
 $$
 what gives~(\ref{eq:I_of_cell}). 
\end{proof}

\section{$V$-manifolds (real orbifolds)}\label{sec:V-manifolds}  

Let us give some definitions in a form appropriate for a discussion below.

For a $G$-space $X$, that is a topological space $X$ with a (left) $G$-action, and for an embedding
$G\subset H$ ($G$ and $H$ are finite groups), let {\em the induction} $\ind_G^H X$ be the $H$-space
defined as the quotient
$$
\ind_G^H X=H\times X/\sim\,,
$$
where $(h_1,x_1)\sim (h_2,x_2)$ if (and only if) there exists $g\in G$ such that $x_1=gx_2$, $h_1=h_2g^{-1}$
(with an obvious $H$-action).
As a topological space $\ind_G^H X$ is the union of several ($\frac{\vert H\vert}{\vert G\vert}$) copies of $X$.
If $X$ is, say, a ($C^{\infty}$-) manifold or a complex quasi-projective variety, the space $\ind_G^H X$ is
of the same type.

\begin{definition}
 An ($n$-dimensional)  uniformizing system on a topological space $X$ is a quadriple
 $(U,\widetilde{U},G,\varphi)$, where $U$ is an open subset of $X$, $G$ is a finite group, $\widetilde{U}$ is
 a smooth ($C^{\infty}$-) $n$-dimensional manifold with a $G$-action, and $\varphi$ is a map $\widetilde{U}\to U$
 such that $\varphi(gx)=\varphi(x)$ (that is $\varphi$ factorizes through a map $p_{\varphi}:\widetilde{U}/G\to U$)
 and the corresponding map $p_{\varphi}$ is a homeomorphism.
\end{definition}

\begin{remark}
 In some cases one adds the condition that the fixed point set of each element of $G$ has codimension at least two
 in $\widetilde{U}$. This restriction is not necessary in this paper and it is more convenient not to require it.
\end{remark}

\begin{definition}
 Two uniformizing systems $(U',\widetilde{U}',G',\varphi')$ and $(U'',\widetilde{U}'',G'',\varphi'')$ on $X$ are
 equivalent if for any point $x\in U'\cap U''$ there exists a neighbourhood $U$ of $x$ in $U'\cap U''$,
 a group $G$ contained both in $G'$ and in $G''$ (that is with embeddings into $G'$ and into $G''$)
 and a uniformizing system $(U,\widetilde{U},G,\varphi)$ such that the $G'$-manifolds $\ind_G^{G'} \widetilde{U}$
 and $(\varphi')^{-1}(U)$ are isomorphic over $U$ (that is by an isomorphism commuting with the projections to $U$)
 and the $G''$-manifolds $\ind_G^{G''} \widetilde{U} $ and $(\varphi'')^{-1}(U)$ are isomorphic over $U$ as well.
\end{definition}

\begin{definition}
 A ($V$-manifold) atlas on a topological space $X$ is a collection of $n$-dimensional uniformizing systems
 $\{(U_{\alpha},\widetilde{U}_{\alpha},G_{\alpha},\varphi_{\alpha})\}$ on $X$ such that
 $\bigcup_{\alpha}U_{\alpha}=X$ and any two uniformizing systems from the collection are equivalent.
\end{definition}

\begin{definition}
 Two atlases on $X$ are  equivalent if their union is an atlas on $X$ as well.
\end{definition}

\begin{definition} (see \cite{Satake1}, \cite{ChenRuan})
 An $n$-dimensional $V$-manifold $Q$ is a 
 separable Hausdorff space $X=X_Q$ with an equivalence class of $n$-dimensional atlases on it.
\end{definition}

One can define in a natural way the notion of a $V$-manifold with boundary: see \cite{Satake2},
\cite[Appendix]{ChenRuan}. In order to ensure that the topological characteristics discussed below
are defined, one has to empose certain finiteness conditions on $V$-manifolds under consideration.
We will assume that in what follows all $V$-manifolds are interiors of compact $V$-manifolds with
boundaries. (For short we will call them {\em tame}.)

The universal Euler characteristic as well as other invariants discussed below can also be regarded as homomorphisms
from the {\em Grothendieck ring} $K_0^{\rm fGr}(\Var)$ {\em of finite group actions} defined in \cite{fGr}.
The Grothendieck ring $K_0^{\rm fGr}(\Var)$ is the Abelian group generated by the classes $[(X,G)]$ of quasi-projective
$G$-varieties for all finite groups $G$ modulo the relations
\begin{enumerate}
 \item[1)] if $(X,G)\cong(X',G')$ (that is if  there exist a group isomorphism $\alpha:G\to G'$ and an (algebraic)
 isomorphism $\psi:X\to X'$ such that $\psi(gx)=\alpha(g)\psi(x)$), then $[(X,G)]=[(X',G')]$;
 \item[2)] if $Y$ is a Zariski closed $G$-subset of $X$, then $[(X,G)]=[(Y,G)]+[(X\setminus Y,G)]$;
 \item[3)] if $G$ is a subgroup of a finite group $H$ and $X$ is a $G$-variety, then $[(X,G)]=[(\ind_G^H X,H)]$.
\end{enumerate}
The multiplication in $K_0^{\rm fGr}(\Var)$ is defined by the Cartesian product of the varieties and of the
groups acting on them.

It is convenient to discuss some properties of the universal Euler characteristic in a purely topological
setting. For that we will consider a sort of nice topological spaces with finite group actions. The notion
of an equivariant $CW$-complex was introduced in \cite{TtD}. An equivariant $CW$-complex with a finite group $G$
action is a $CW$-complex possessing, in particular, the following property: if $g\sigma=\sigma$ for  a cell
$\sigma$ of the complex, then $g_{\vert\sigma}$ is the identity transformation.

\begin{definition}
 An  equivariant cell complex is an  invariant locally closed union of cells in a finite equivariant
 (with respect to a finite group) $CW$-complex.
\end{definition}

A quasi-projective $G$-variety or a (real) semialgebraic $G$-set ($G$ is a finite group) can be represented as
an equivariant cell complex. For an equivariant cell complex its Euler characteristic, equivariant Euler
characteristic, orbifold Euler characteristic, {\dots}  are well defined: see below.

\section{Additive invariants of $V$-manifolds}\label{sec:additive_of_V-manifolds}

There are a number of additive invariants defined for $V$-manifolds.

For a $G$-space $X$ ($G$ is a finite group) and for a point $x\in X$, let $G_x=\{g\in G:gx=x\}$ be the isotropy
subgroup of $x$. For a subgroup $H\subset G$, let $X^H=\{x\in X: hx=x{\text{ for all }}h\in H\}$ be the fixed
point set of the subgroup $H$ and let $X^{(H)}=\{x\in X: G_x=H\}$ be the subspace of points with the isotropy
subgroup $H$ ($X^{(H)}\subset X^H$). For a conjugacy class $[H]$ of subgroups of $G$, let
$X^{([H])}=\{x\in X: G_x\in[H]\}$. Let $\G$ be the set of the isomorphism classes of finite groups.

Let $Q$ be a (tame) $V$-manifold. For each point $x\in Q$ one associates the isotropy (sub)group $G_x$.
For a finite group $G$, let $Q^{(G)}=\{x\in Q: G_x\cong G\}$. One can see that the $V$-manifold $Q^{(G)}$
is a global quotient (under an action of the group $G$). Moreover, its reduction is the usual ($C^{\infty}$-)
manifold (with the action of the trivial group).

The {\em Euler-Satake characteristic} of $Q$ (\cite{Satake2}) is defined by
\begin{equation}\label{eq:Euler-Satake}
\chi^{\rm ES}(Q)=\sum_{\{G\}\in \G}\frac{1}{\vert G\vert}\chi(Q^{(G)}).
\end{equation}

The {\em orbifold Euler characteristic} (defined in, e.~g., \cite{AS}, \cite{HH}) can be defined for a
$V$-manifold  by 
\begin{equation*}
\chi^{\rm orb}(Q)=\sum_{\{G\}\in \G}\chi^{\rm orb}(G/G,G)\cdot \chi(Q^{(G)})\,,
\end{equation*}
where $\chi^{\rm orb}(G/G,G)$ is the orbifold Euler characteristic of the one-point $G$-set $G/G$ (in the
sense of \cite{AS}, \cite{HH}). If all the isotropy groups of points of $Q$ are Abelian, one has
\begin{equation*}
\chi^{\rm orb}(Q)=\sum_{\{G\}\in \G}\vert G\vert\cdot \chi(Q^{(G)})\,.
\end{equation*}

The {\em higher order} orbifold {\em Euler characteristics} $\chi^{(k)}(X,G)$ of a $G$-space $(X,G)$
were defined in \cite{BrF}, \cite{Tamanoi}. For $k=0,1$, one has $\chi^{(0)}(X,G)=\chi(X/G)$,
$\chi^{(1)}(X,G)=\chi^{\rm orb}(X,G)$. (We follow the numbering used in \cite{Tamanoi}.) For a $V$-manifold
they can be defined by
\begin{equation*}
\chi^{(k)}(Q)=\sum_{\{G\}\in \G}\chi^{(k)}(G/G,G)\cdot \chi(Q^{(G)})\,.
\end{equation*}
If all the isotropy groups of points of $Q$ are Abelian, one has
\begin{equation*}
\chi^{(k)}(Q)=\sum_{\{G\}\in \G}\vert G\vert^k\cdot \chi(Q^{(G)})\,.
\end{equation*}

One can see that the Euler-Satake characteristic~(\ref{eq:Euler-Satake}) can be regarded as the Euler characteristic
of order $(-1)$. This fits to the definition of the $\Gamma$-Euler-Satake characteristic $\chi^{\rm ES}_{\Gamma}(Q)$
of a $V$-manifold for a group $\Gamma$ in \cite{Farsi+}: for $\Gamma=\Z^{k+1}$ one gets the Euler characteristic
of order $k$; for $\Gamma=\{1\}$ (i.~e., $\Gamma=\Z^0$), one gets the Euler--Satake characteristic.

All these characteristics possess the additivity and the multiplicativity properties: if $Q'$ is a (closed)
$V$-submanifold of a $V$-manifold $Q$, one has $\chi^{\bullet}(Q)=\chi^{\bullet}(Q')+\chi^{\bullet}(Q\setminus Q')$;
if $Q_1$ and $Q_2$ are $V$-manifolds, one has
$\chi^{\bullet}(Q_1\times Q_2)=\chi^{\bullet}(Q_1)\cdot \chi^{\bullet}(Q_2)$. (Here $\chi^{\bullet}$ means
$\chi^{\rm ES}$, $\chi^{\rm orb}$, \dots)

All these invariants can be defined on the Grothendieck ring $K_0^{\rm fGr}(\Var)$ of quasi-projective varieties with
finite  groups actions so that they are ring homomorphism from $K_0^{\rm fGr}(\Var)$ to the ring $\Z$ of integers.
Moreover, all of them can be defined for equivariant cell complexes. For example, if $X$ is an equivariant cell
complex with an action of a finite group $G$, then its orbifold Euler characteristic can be defined by the equation
$$
\chi^{\rm orb}(X,G)=\frac{1}{\vert G\vert}\sum_{{(g,h)\in G\times G:}\atop{\\gh=hg}}\chi(X^{\langle g,h\rangle})\,,
$$
where $\langle g,h\rangle$ is the subgroup of $G$ generated by $g$ and $h$, or by the equation
$$
\chi^{\rm orb}(X,G)=\sum_{[H]\in{\tiny{\conjsub}} G}\chi(X^{([H])}/G)\chi^{\rm orb}(G/H,G)\,,
$$
where $\conjsub G$ is the set of conjugacy classes of subgroups of $G$.

\section{The universal Euler characteristic}\label{sec:Universal_Euler}

Let $\GG$ be the set of al  isomorphisms classes of finite groups and let $\RR$ be 
be the free Abelian group generated by the  elements $T^G$ correspondig to the classes  $\{ G\}\in \RR$.
We will write an element of $\RR$ as a finite sum of the form 
$\sum_{ \{G\}\in \G} a_G T^G$, where $a_G\in \Z$. One has a natural multiplication  on $\RR$ defined by
$T ^{G'}\cdot T^{G''}=T^{G' \times G''}$. 
Thus $\RR$ is a ring. Acording to the Krull-Schmidt theorem each finite group has a unique
 representation as the   product of indecomposable finite groups. 
 Let $\GG_{ind}$ be the set of  
 the isomorphisms  classes of indecomposable  finite groups.
The Krull-Schmidt theorem implies that $\RR$ is the polynomial ring $\Z [T_G ]$ in the variables $T_G$ corresponding to 
(the isomorphisms  classes of) the indecomposable finite groups. 
(If a finite group $G$ has the decomposition $G\cong \prod_{i=1}^r G{(i)}$ with indecomposable $G{(i)}$, one has 
$T^G=\prod_{i=1}^r T_{G{(i)}}$.) 

  \begin{definition}\label{def:universal-euler}
The universal Euler characteristic of  (tame)   $V$-manifold $Q$
is defined by
\begin{equation}\label{eq:universal-euler}
 \chiun(Q)=\sum_{\{G\}\in \GG} \chi(Q^{(G)}) T^G \in \RR.
\end{equation}
  \end{definition}

One can see that $\chiun$ is an additive and multiplicative invariant of $V$-manifolds.

Another interpretation of the ring $\RR$ is the following one:
it is the subring of the Grothendieck ring ${K}_0^{\rm fGr}(\Var)$ of quasi-projective varieties
with finite groups actions generated by the finite sets, i.~e., by zero-dimensional varieties.
In terms of the descrption/definition of the  Grothendieck ring ${K}_0^{\rm fGr}(\Var)$ in \cite{fGr},
an element $a=\sum_{\{G\}\in \GG} a_G T^G \in \RR$ can be  represented by the (virtual) set consisting of
$\sum_{G\in \GG} a_G$ points
so that on $a_G$ of them ($a_G$ may be negative) one has the trivial action of the group $G$.
In terms of the description given above, an element $a$ can be represented by a pair $(X, G^*)$,
where $G^*$ is a group containing all the groups $G$ with $a_G\ne 0$ and $X$ is the union over $\{G\}\in \GG$
of the (finite) sets which consists of $a_G$ copies of $G^*/G$ with the natural $G^*$-action. 

The universal Euler characteristic
$ \chiun(\bullet) $ can be defined for elements of the Grothendieck ring  
${K}_0^{\rm fGr}(\Var)$ 
of quasi-projective varieties  with  finite groups actions. 
Moreover it can be defined for equivariant cell complexes.
For an equivariant cell complex $X$ with an action of a finite group $G$,
$\chiun(X,G)$ can be defined by the equation
$$
\chiun(X,G)=\sum_{[H]\in \tiny{\conjsub} G} \chi(X^{([H])}/G) T^H.
$$
In orther terms it can be defined in the following way.
The set $\Sigma_k$ of cells of dimesion $k$ in $X$   is a finite set  
with a $G$-action. Then  
$$ \chiun(X,G)=\sum_{k} (-1)^k [(\Sigma_k, G)]\in \RR.$$

One can see that  the universal Euler characteristic of equivariant  cell complexes possesses
the following properties:
\begin{enumerate}
 \item[1)] additivity: if $(Y;G)$ is a closed $G$-invariant subcomplex of $(X,G)$, then
$$
\chiun(X,G)=\chiun(Y,G)+\chiun(X\setminus Y ,G)\,;
$$
 \item[2)] multiplicativity: if $(X', G')$ and $(X'', G'')$
are two equivariant cell complexes, then
$$
\chiun(X'\times X'' , G'\times G'')=\chiun(X', G')\cdot \chiun(X'', G'')\,;
$$
 \item[3)] the induction relation: if $G$ is a subgroup of $H$, then 
$$\chiun(\ind_G^H X,H) = \chiun(X,G).$$
\end{enumerate}

The relations 1) and 3) permit to define the universal Euler characteristic $\chiun$
as a function on the Grothendieck ring 
${K}_0^{\rm fGr}(\Var)$ (with values in $\RR$).
The relations 1) and 2) mean that it is a ring homomorphis to $\RR$.  

Let us give a statement which explains the word universal in the name of $\chiun$.

\begin{theorem}\label{thm:ring hom}
If $I$ is an additive invariant of (tame) $V$-manifolds with values in an Abelian group $R$, then there exists
a unique  homomorphism of Abelian groups $r: \RR \to R$ such that $I(\bullet)=r (\chiun(\bullet))$.
If $R$ is a ring and $I$ multiplicative, then $r$ is a ring homomorphism.
\end{theorem}

\begin{proof}
Let $Q$ be a tame $V$-manifold. For a finite group (or rather for an isomorphism class of 
finite groups) $G$, let $Q^{(G)}$  be the set of points $x\in X_Q$ with isotropy group
isomorphic to $G$.
The fact that $Q$ is assumed to be tame implies that there are  finitely many classes 
$G$ such that $Q^{(G)}\ne \emptyset.$
($Q^{(G)}$ is a non-closed $V$-submanifold of $Q$ whose reduction is a usual 
$C^{\infty}$-manifold.)

The set $\GG$ of isomorphism classes of finite groups
is a partially order  set. Let $G$ be a minimal element
from $\GG$ with  $Q^{(G)}\ne \emptyset.$
By additivity one has $I(Q)=I(Q^{(G)})+ I(Q\setminus Q^{(G)}).$
Itereting this equation one gets $I(Q)=\sum_{G\in  \GG} I(Q^{(G)}).$
Since $Q^{(G)}$  is the usual $C^{\infty}$-manifold, due to Proposition~\ref{prop:unique_for_manifolds}
one has $I(Q^{(G)})= \chi(Q^{(G)})\tau_G$,
with $\tau_G=I(T^G)\in R$. 
One can see that the group homomorphis $r: \RR \to R$ which sends 
the universal Euler characteristic $\chiun(\bullet)$ to  $I(\bullet)$ is defined by $r(T^G)=\tau_G$.

The multiplicativity of $r$, for $I$ being multiplicative, is obious.
\end{proof}

One also has the following universality properties of $\chiun(\bullet)$ for equivariant cell complexes.

\begin{theorem}\label{thm:ring hom2}
Let $I$ be an additive invariant of equivariant cell complexes with values in an Abelian group $R$
posseing the induction property: $I(\ind_G^H X,H) = I(X,G)$ for finite groups $G\subset H $.
Then there exists a unique 
homomorphism of Abelian groups $r: \RR \to R$ such that $I(\cdot,\cdot))=r (\chiun(\cdot,\cdot))$.
If $R$ is a ring and $I$  multiplicative, then $r$ is a ring homomorphism.   
\end{theorem}

\begin{proof}
Let $\Sigma$ be the (finite set) of cells in an equivariant cell complexes $(X,G)$ ($\Sigma$ is a $G$-set).
The  additivity property of $I$ imply that: 
$I(X,G)=\sum_{[\sigma] \in \Sigma/G} I(G\sigma ,G)$ where $\sigma= \sigma^k$ is an open cell (of certain
dimension $k$) in $X$: a representative of the orbit $[\sigma]$,
$G\sigma$ is the union $\cup_{g \in G} g \sigma$  of the $G$-shifts of $\sigma$.

Let $G_{\sigma^k}$ be the isotropy group of the cell $\sigma^k$.(Let us recall that $G_\sigma$ acts trivialy on
$\sigma$.) One has $G \sigma^k= (G/G_{\sigma^k})\times \sigma^k$, where $G/G_{\sigma^k}$ is a finite $G$-set.
Just as in the proof of Proposition~\ref{prop:unique_for_manifolds} one has
$$
I(G \sigma^k, G)=(-1)^k I((G/G_{\sigma^k})\times \{pt\},G).
$$
The induction property implies that
$I((G/G_{\sigma^k})\times \{pt\},G)= I(G_{\sigma^k}/G_{\sigma^k})\times \{pt\},G_{\sigma^k}).$
Let us denote $I(G/G\times \{pt\},G)$ by $\tau_G$. One can see that group homomorphism $r:\RR\to R$
which sends $\chiun(\cdot,\cdot)$ to $I(\cdot,\cdot)$ is defined by $r(T^G)=\tau_G$.
The multiplicativity of $r$ in the case when $I$ is multiplicative is obvious. 
\end{proof}

\section{
$\lambda$-ring structures on $\RR$ and the correspondig  power structures}\label{sec:lambda-power} 

A way to formulate  an analogue of Macdonald type equations for an (additive and multiplicative)
invariant with values in a ring $R$ is through the so called \emph{power structure}
over the ring $R$: \cite{GLM-MRL}.
A power structure over a ring $R$ is a method to give sense to an expression of the form
$(1+a_1 t+ a_2 t^2+\ldots)^m$  with $a_i, m\in R$ as a power series form 
$1+t R[[t]]$ so that all properties of the usual exponential function hold.
The notion of a power structure over a ring $R$ is related with the notion 
of a $\lambda$-ring (some times called a pre-$\lambda$-ring) 
structure on $R$: \cite{Knutson}.

We will describe two 
$\lambda$-ring structures on the ring $\RR$ appropriate 
for the formulation of Macdonald type equations for the symmetric products and for the configuration spaces.
As it was explained, $\RR$ is the ring of polynomials in the variables $T_G$ corresponding
to the isomorphism classes $G$ of indecomposable finite groups.
The standard $\lambda$-ring structure on the polynomial ring $\Z [x_1, x_2, x_3, \ldots ]$
(see, e.~g., \cite{Knutson}) is defined in the following way: for 
$$p(\overline{x})=\sum_{\overline{k}} p_{\overline{k}} 
{\overline{x}}^{\overline{k}} 
\in \Z [x_1, x_2, x_3, \ldots ],$$
one has 
\begin{equation}\label{eq:lambda-polynomial}
\lambda_{p(\overline{x})}(t)
=\prod_{\overline{k}} \left(1- \overline{x}^{\overline{k}} t  
 \right)^{-p_{\overline{k}}}.
\end{equation}
Equation (\ref{eq:lambda-polynomial}) follows directly   
 from the equation for the $\lambda$-series corresponding to a monomial:
$$
\lambda_{\overline{x}^{\overline{k}}}(t)=
 \left(1- \overline{x}^{\overline{k}} t   \right)^{-1}.
$$

Natural $\lambda$-rings structures on $\RR$ are different ones.
To define a $\lambda$-ring structure on $\RR$, one can define the $\lambda$-series, say, $\nu_{T^G}(t)$
for a monomial $T^G$ from $\RR$. Namely, if the series $\nu_{T^G}(t)$ is defined for all $\{G\}\in\G$
(so that $\nu_{T^G}(t)=1+T^Gt+\ldots$), then one defines the $\lambda$-series for an element
$A=\sum_{\{G\}\in \GG} a_G T^G\in \RR$ ($a_G\in \Z$) by
$$
\nu_{A}(t)=\prod_{\{G\}\in \GG} \left(\nu_{T^G}(t) \right)^{a_G}.
$$

Let us first describe the $\lambda$-rings structure on $\RR$
corresponding to the symmetric products of spaces. We will call it the {\em power product $\lambda$-structure}.
This structure will be defined defined by a $\lambda$-series $\zeta_{\bullet}(t)$.
For a finite group $G$, let $G_n=G\wr S_n=G^n\rtimes S_n$ be the corresponding wrearth product. 
Let us  define $\zeta_{T^G}(t)$ for the monomial $T^G$ by the equation 
\begin{equation}\label{eq:lambba-inde1}
 \zeta_{T^G}(t) =1+ \sum_{n=1}^{\infty} T^{G_n} t^n.
\end{equation}
In particular,
$$
 \zeta_{1}(t) =1+ \sum_{n=1}^{\infty} T^{S_n} t^n.
$$

\begin{remark}
 Pay attention that the coefficients of the $\lambda$-series
for a monomial are monomials as well.
\end{remark}

Let us (partially) describe the series (\ref{eq:lambba-inde1}) in terms of the variables $T_G$
for the polynomial ring $\RR$ ($G$ runs through isomorphism classes of indecomposable finite groups).
For such a description  the variables coresponding to the Abelian groups play a special role.
Let $A_{p,k}\cong \Z_{p^k} ( p$ is prime, $k\geq 1$) be the indecomposable  finite Abelian groups.
For a group  $G=\prod_{p,k} (A_{p,k})^{l_{p,k}} \prod G(i)^ {k_i}$ with non-Abelian indecomposable 
finite groups $G(i)$,  for  $n>1$, one has
$$
(\prod_{p,k} A_{p,k}^{l_{p,k}} \prod G(i)^ {k_i})_n \cong \left(\prod_{p,k: p \centernot |  n} 
 A_{p,k}^{l_{p,k} } \right)\times  \widehat{G} (n),
$$
where $\widehat{G}(n)$ is a indecomposable non-Abelian group (depending on  the group $G$ of course)
(see \cite{Neu} and also \cite{DF} for more precise statements)
and therefore
\begin{equation}\label{eq:lambba-inde}
 \zeta_{\prod_{p,k} T_{A_{p,k}}^{l_{p,k}} \prod_{i} T_{G(i)}^{k_i}}(t) 
=1+ \left(\prod_{p,k} T_{A_{p,k}}^{l_{p,k}} \prod_{i} T_{G(i)}^{k_i}\right) t
+ \sum_{n=2}^{\infty} \left( T_{\widehat{G}(n)}\prod_{p,k: p \centernot | n} T_ {A_{p,k}}^{l_{p,k}}\right) t^n\,.
\end{equation}

The described $\lambda$-ring structure on $\RR$ defines (in the usual way: see \cite{GLM-MRL}, \cite{GLM-Stek})
a power structure over $\RR$. We will call it the {\em symmetric product power structure}.
Let us recall that according to the construction of the power structure one has
$$
\left(\zeta_1(t)\right)^{T^G}=\zeta_{T^G}(t)\,.
$$

The other  $\lambda$-ring structure on $\RR$ corresponds to the configuration space of spaces.
We will call it the {\em configuration space $\lambda$-ring structure}.
This structure will be  defined by a $\lambda$ -series $\lambda_{\bullet}(t)$. 
As above it is suffient to define this series for monomials. Let
\begin{equation}\label{eq:lambba-lambda}
 \lambda_{T^G}(t) =1+ T^G t\,.
\end{equation}
In particular, $\lambda_1(t)=1+t$.

This $\lambda$-ring structure on $\RR$ defines the corresponding {\em configuration space power structure} over $\RR$.
The described power structures (the symmetric product and the configuration space ones) over $\RR$ are different: 
see computations in \cite[page~17]{fGr}.

From \cite{fGr} and the interpretation of $\RR$ given  above it follows that the configuration space
power structure over $\RR$ is effective in the following sense. Let $\RR_{+}$ be the subsemiring of $\RR$ 
consisting of polynomials in $T_G$ with non-negatives coefficients. 
The effectiveness of the power structure means that if $a_i$ and $m$ are from
$\RR_{+}$, then all the coefficientes  of the series $(1+a_1 t+ a_2 t^2+\ldots)^m$ belong to
$\RR_{+}$ as well.

\begin{remark}
The fact that this power structure is effective is not a direct consequence of  the equation~(\ref{eq:lambba-lambda})
for the $\lambda$-series. The effectiveness of the configuration space  power structure is a consequence
of an explicit equation for it: see~\cite[Equation~10]{fGr}.
The symmetric product power structure over $\RR$ is not effective: see again  
\cite[page~17]{fGr}.
\end{remark}

\section{Macdonald type equations for the universal Euler characteristics for symmetric products}
\label{sec:Macdonald} 

Let $Q$ be a (tame) $V$-manifold. The $n$-th symmetric power $S^n Q$ of $Q$ is the $V$-manifold defined in the following way.
The underline space of $S^n Q$ is the $n$-th symmetric power
$S^n X_Q$ of the underline space $X_Q$ of the V-manifold $Q$.
Let
$$
\underline{x}=(x_1,\ldots, x_1,\ldots, x_s, \ldots, x_s)
$$
where $x_i$, $i=1,\ldots, s$, is a point of $X_Q$ with the multiplicity $k_i$ in $\underline{x}$,
$\sum_{i=1}^s k_i=n$,  $x_i\neq x_j$ for $i\ne j$, be a point  of $S^n X_Q$
and let $(U_i, \tilde{U}_i, G(i), \varphi_i)$ be local uniformizing systems for neighbourhoods $U_i$ of
the points $x_i$ such that
$U_i\cap U_j =\emptyset$ for $i\ne j$. Then the orbifold structure on $S^n Q$ in a neighbourhood of $\underline{x}$ is defined by the
local uniformizing system
$(S^{k_1} U_1\times\ldots\times  S^{k_s} U_s, \tilde{U}_1^{k_1}\times \ldots\times  \tilde{U}_s^{k_s}, 
(G(1))_{k_1}\times \ldots \times (G(s))_{k_s}, \bar{\varphi} )$, where
$(G(i))_{k_i}$ is the wreath product $G(i)\wr S_{k_i}$ acting on the Cartesian power $U_i^{k_i}$ in the usual way,
$\bar\varphi=(\varphi_1\times\ldots\times\varphi_1\times\ldots\times\varphi_s\times\ldots\times\varphi_s)$.

\begin{theorem}\label{theo:Mac_V-manifolds} For a $V$-manifold $Q$ one has
\begin{equation}\label{eq:Mac_V-manifolds}
 1+\sum_{n=1}^{\infty}\chiun(S^nQ)t^n=\zeta_{\chiun(Q)}(t)=(\zeta_1(t))^{\chiun(Q)},
\end{equation}
where the right hand side is written in terms of the symmetric product power structure over $\RR$. 
\end{theorem}

Let us recall that, if $\chiun(Q)=\sum_{\{G\}\in \GG} a_G T^G$, then
$$
\zeta_{\chiun(Q)}(t)=\prod_{\{G\}\in \GG} \left(1+ T^Gt + T^{G_2} t^2+ T^{G_3} t^3+\ldots\right)^{a_G}.
$$

\begin{proof}
 Let us denote the left hand side of Equation (\ref{eq:Mac_V-manifolds}) by $\xi_Q(t)$.
 If $Q'$ is a closed $V$-submanifold of $Q$, one has $\xi_Q(t)=\xi_{Q'}(t)\xi_{Q\setminus Q'}(t)$.
 (This follows from the fact that $S^nQ$ is the disjoint union of $S^kQ'\times S^{n-k}(Q\setminus Q')$
 for $0\le k\le n$.) The representation of $Q$ as the disjoint union of the sub-$V$-manifolds $Q^{(G)}$
 for $G\in\G$ permits to prove the statement for $Q=MT^G$ for a $C^{\infty}$-manifold $M$ (with the action
 of the trivial group). A representation of $M$ as a cell complex permits to prove the statement for
 $Q=\sigma^kT^G$, where $\sigma^k$ is an open cell of dimension $k$.
 The fact that a $k$-dimensional cell can be represented as the union of two $k$-dimensional cells and
 one $(k-1)$-dimensional cell implies that
 $$
 \xi_{\sigma^kT^G}(t)=\left(\xi_{\sigma^0T^G}(t)\right)^{(-1)^k}.
 $$
 Therefore it is sufficient to show~(\ref{eq:Mac_V-manifolds}) for $Q=\sigma^0T^G$. In this case we have
 $$
 \xi_{\sigma^0T^G}(t)=1+T^Gt+T^{G_2}t^2+T^{G_3}t^3+\ldots=\zeta_{T^G}(t)=\left(\zeta_1(t)\right)^{T^G}=
 \left(\zeta_1(t)\right)^{\chiun(\sigma^0T^G)}.
 $$ 
\end{proof}

One has a Macdonald type equation for equivariant cell complexes (and therefore for representatives $(X,G)$
of elements of the Grothendieck ring $K_0^{\rm fGr}(\Var)$). For a cell complex $(X,G)$, let $(X^n,G_n)$
be the Cartesian power of the complex $X$ with the standard action of the wreath product $G_n$.

\begin{theorem}\label{theo:Mac_cells} For an equivariant cell complex $(X,G)$, one has
$$
 1+\sum_{n=1}^{\infty}\chiun(X^n, G_n)t^n=\zeta_{\chiun(X,G)}(t)=(\zeta_1(t))^{\chiun(X,G)}.
$$
(the right hand side is in terms of the symmetric product power structure).
\end{theorem}

The {\bf proof} is essentially the same as the one of Theorem~\ref{theo:Mac_V-manifolds} with the only difference
that the general case is reduced not to $\sigma^kT^G$, but to $\sigma^k\times (G/G_{\sigma^k})$.

\section{Macdonald type equations for the universal Euler characteristics for configuration spaces}
\label{sec:Macd-confg} 

For a $V$-manifold $Q$, its $n$-th configuration space $M_n Q$ is the $V$-manifold defined in the following way.
Its  underline space is $M_n X_Q=(X_Q^n \setminus \Delta )/S_n\subset S^n X_Q$; the $V$-manifold structure 
on it comes from the one in $S^ n X_Q$. (Pay attention that one has to define 
local uniformizing systems only for points 
$
\underline{x}=(x_1,\ldots, x_n)
$
with $x_i\ne x_j$ if $i\ne j$.)

\begin{theorem}\label{theo:Mac_V-manifolds-conf} For a $V$-manifold $Q$, one has
\begin{equation}\label{eq:Mac_V-manifolds-conf}
 1+\sum_{n=1}^{\infty}\chiun(M_n Q)t^n=\lambda_{\chiun(Q)}(t)=(1+t)^{\chiun(Q)}.
\end{equation}
where the right hand side is written in terms of the configuration space  power structure over $\RR$. 
\end{theorem}

Let us recall that $\lambda_{1}(t)=1+t$, if $\chiun(Q)=\sum_{\{G\}\in \GG} a_G T^G$, then
$$
\lambda_{\chiun(Q)}(t)=\prod_{\{G\}\in \GG} \left(1+ T^Gt\right)^{a_G}.
$$

An analogue of this equation for equivariant cell complex is given in the following statement.
For equivariant cell complex  $(X,G)$ and for $n\geq 1$, let $\Delta_G \subset X^n$ be the big $G$-diagonal 
in $X^n$ consisting of (ordered) n-tuples  $(x_1,\ldots, x_n)\in X^n$ with at least two of the components $x_i$
from the same $G$-orbit.

\begin{theorem}\label{theo:Mac_cells-conf} For an equivariant cell complex $(X,G)$, one has
$$
 1+\sum_{n=1}^{\infty}\chiun(X^n\setminus \Delta_G, G_n)t^n=\lambda_{\chiun(X,G)}(t)=(1+t)^{\chiun(X,G)}.
$$
(the right hand side is in the terms of the configuration space power structure).
\end{theorem}
     
Proofs of Theorems~\ref{theo:Mac_V-manifolds-conf} and~\ref{theo:Mac_cells-conf} are minor modifications
of those of Theorems~\ref{theo:Mac_V-manifolds} and~\ref{theo:Mac_cells}.

\end{document}